\documentclass[11pt,a4paper]{article}
\usepackage{amsmath}
\usepackage{amssymb}
\usepackage{amsthm}
\usepackage{amsfonts}
\usepackage{latexsym}
\usepackage{verbatim} 
\usepackage{xcolor}
\usepackage[a4paper,width=16cm,top=2cm,bottom=2cm,footskip=1cm
]{geometry}
\usepackage[flushmargin]{footmisc}
\usepackage{color}                    
\usepackage{esvect}
\usepackage{mathrsfs}
\theoremstyle{plain}
\newtheorem{thm}{Theorem}[section]

\newtheorem{lem}{Lemma}[section]

\theoremstyle{remark}

\theoremstyle{definition}
\newtheorem{defn}{Definition}[section]

\newtheorem{rem}{Remark}[section]

\newcommand{\ddbar}{\overline\partial}
\newcommand{\pr}{\partial}
\newcommand{\ol}{\overline}

\newcommand{\norm}[1]{\left\Vert#1\right\Vert}
\newcommand{\abs}[1]{\left\vert#1\right\vert}
\newcommand{\set}[1]{\left\{#1\right\}}
\newcommand{\To}{\rightarrow}

\title{Embedding theorems for  quantizable pseudo-K\"ahler manifolds}
\author{Andrea Galasso\footnote{\noindent{\bf Address:} National Center for Theoretical Sciences, Room 407, Chee-Chun Leung Cosmology Hall, National Taiwan University; {\bf current address:} Dipartimento di Matematica e Applicazioni, Universit\`a degli Studi di Milano-Bicocca, Via R.	Cozzi 55, 20125 Milano, Italy; {\bf ORCID iD:} 0000-0002-5792-1674; {\bf e-mail}: andrea.galasso@unimib.it; andrea.galasso.91@gmail.com}\,\, and Chin-Yu Hsiao\footnote{\noindent{\bf Address:} Institute of Mathematics, Academia Sinica, 6F, Astronomy-Mathematics Building, No.1, Sec.4, Roosevel: {\bf ORCHID iD:} 0000-0002-1781-0013; {\bf email}: chsiao@math.sinica.edu.tw; chinyu.hsiao@gmail.com}}

\date{}

\begin{document}
	\maketitle
	
	\begin{abstract}  Given a compact quantizable pseudo-K\"ahler manifold $(M,\omega)$ of constant signature, there exists a Hermitian line bundle $(L,h)$ over $M$ with curvature $-2\pi i\,\omega$. We shall show that the asymptotic expansion of the Bergman kernels for $L^{\otimes k}$-valued $(0,q)$-forms implies more or less immediately a number of analogues of well-known results, such as Kodaira embedding theorem and Tian's almost-isometry theorem. \end{abstract}
	
	\begin{center}
		\textit{In memory of Steve Zelditch}
	\end{center}
	
	\tableofcontents
	
	
	
	\section{Introduction}\label{s-gue220919yyd}
	
	Given a compact symplectic manifold $(M,\omega)$ such that $[\omega] \in H^2(M,\mathbb{Z})$ we can define a Hermitian line bundle $(L,h^L)$ having a unique holomorphic connection whose curvature is $R^L=-2\pi i\,\omega$. One can always pick an almost complex structure $J$ compatible with $\omega$, in \cite{sz} under the assumption of $L$ being positive a space of almost holomorphic sections $H^0_J(M,L^{\otimes k})$ is defined and the authors prove symplectic analogues of standard results in complex geometry; particularly those involving the zero sets of sections and the corresponding maps they define into complex projective space. As far as the authors are aware, if the positivity assumption on $L$ is dropped no such results are known and in fact one can not expect the dimensions of $H^0_J(M,L^{\otimes k})$ to grow as in the positive case as $k$ goes to infinity in general. From the other hand, since the curvature of the line bundle $L$ in non-degenerate, suppose that it has constant signature $(n_-,n_+)$, one can try to replace the spaces $H^0_J(M,L^{\otimes k})$ with spaces of harmonic $L^{\otimes k}$-valued $(0,n_-)$-forms, this is suggested by results contained in \cite{hsma} and \cite{mm} concerning the asymptotic expansion of the Bergman kernel for $(0,q)$ forms for high tensor powers of $L$ (see also \cite{hsiao} and \cite{hsma2} for the Szeg\H{o} kernels for $(0,q)$ forms). The aim of this paper is to define embeddings of $(M,\omega)$ which are asymptotically symplectic; since the results in \cite{hsma} are presented in the holomorphic category we focus on the case when $J$ is an integrable structure. %; we believe such results are true in the almost complex category as well, and we plan to investigate these on a future work. 
	We now recall some known results and state the main theorems.
	
	We refer to Section \ref{sec:pre} for notations and preliminaries. 
	Let $(L,h^L)\To M$ be a holomorphic line bundle, where $M$ is a compact complex manifold of complex dimension $n$ and $h^L$ is the Hermitian metric of $L$. Let $R^L$ be the curvature of $L$ induced by $h^L$. In this work, we assume that $R^L$ is non-degenerate of constant signature $(n_-,n_+)$ on $M$, $n_-+n_+=n$.
	\begin{rem}
		Suppose $M$ is connected. Notice that when $\omega$ is non-degenerate we can prove that $\omega$ has constant signature. For each point $m$ in $M$ consider the continuous map $m \mapsto \lambda_j(m)$ which is the $j$-th eugenvalue of $\omega_m$. By absurd suppose that there exists $x,\,y$ in $M$ and $j=1,\dots,n$ such that $\lambda_j(x)>0$ and $\lambda_j(y)<0$. Since $\lambda_j$ is continuous, by the intermediate value theorem there exists $z$ such that $\lambda_j(z)=0$, which contradicts the hypothesis.
	\end{rem}
	
	Fix a Hermitian metric $\langle\,\cdot\,|\,\cdot\,\rangle$ on the holomorphic tangent bundle $T^{1,0}M$. 
	The Hermitian metric $\langle\,\cdot\,|\,\cdot\,\rangle$ on $T^{1,0}M$ induces a Hermitian 
	metric on $TM\otimes\mathbb C$ and also on $$\bigoplus_{q\in\mathbb N\cup\set{0},0\leq q\leq n}T^{*0,q}M,$$ where $T^{*0,q}M$ is the bundle of $(0,q)$ forms. Consider the vector bundle $T^{*0,q}M\otimes L^{\otimes k}$ whose space of smooth sections is denoted by $\Omega^{0,q}(M,L^{\otimes k})$. Let $\{\overline e_j\}^n_{j=1}$ be a basis of $T^{0,1}M:=\overline{T^{1,0}M}$ on an open set $D$ of $M$. 
	For any strictly increasing multi-index $I=(i_1,\dots,\,i_q)$, $1\leq i_j\leq n$, we denote \[\overline{e}^I:=\overline{e}_{i_1}\wedge \cdots \wedge\overline{e}_{i_q}\,.\] 
	On $D$, we can write a $(0,q)$-form $f\in\Omega^{0,q}(M,L^{\otimes k})$ as follows
	\begin{equation} \label{eq:form}
		f_{|D}= \sum^{<}_{\lvert I\rvert= q} f_I(z)\,\overline{e}^{I}\,, 
	\end{equation}
	where $f_I\in\mathcal{C}^\infty(D,L^k)$, for each $I$ and $\sum^{<}$ means that the summation is performed only over strictly increasing multi-indices. In this paper all multi-indices will be supposed to be
	strictly increasing. The Hermitian metrices $\langle\,\cdot\,|\,\cdot\,\rangle$ and $h^L$ induce a new Hermitian metric $\langle\,\cdot\,|\,\cdot\,\rangle_{h^{L^{\otimes k}}}$ on $T^{*0,q}M\otimes L^{\otimes k}$, the corresponding norm is denoted by $|\cdot|_{h^{L^{\otimes k}}}$. We can define the $L^2$-inner product as follows:
	\[(s_1,s_2)_k=\int_M \langle\,s_1\,|\,s_2\,\rangle_{h^{L^{\otimes k}}}\, \mathrm{dV}_M\quad s_1,\,s_2\in \Omega^{0,q}(M,L^{\otimes k})\,,  \]
	and we denote by $\lVert \cdot \rVert_{k}$ the corresponding norm, where $\mathrm{dV}_M$ is the volume form on $M$ induced by $\langle\,\cdot\,|\,\cdot\,\rangle$.  
	Let $L^2_{0,q}(M,L^{\otimes k})$ be the completion of $\Omega^{0,q}(M,L^{\otimes k})$ with respect to $(\cdot\,,\cdot)_k$. Let $\square^{(q)}_k: \Omega^{0,q}(M,L^{\otimes k})\To\Omega^{0,q}(M,L^{\otimes k})$ be the Kodaira Laplacian. We denote by the same symbol the $L^2$ extension of $\square^{(q)}_k$: 
	\[\square^{(q)}_k: {\rm Dom\,}\square^{(q)}_k\subset L^2_{0,q}(M,L^{\otimes k})\To L^2_{0,q}(M,L^{\otimes k}),\] where 
	${\rm Dom\,}\square^{(q)}_k=\{u\in L^2_{0,q}(M,L^{\otimes k});\, \square^{(q)}_ku\in L^2_{0,q}(M,L^{\otimes k})\}$. The projection 
	\[ P^{(q)}_k\,:\, L^2_{0,q}(M,L^{\otimes k}) \rightarrow H^q(M,L^{\otimes k}):={\rm Ker\,}\square^{(q)}_k \]
	onto the kernel of $\square^{(q)}_k$ is called Bergman projector or Bergman projection. Since $\square^{(q)}_k$ is elliptic, its distributional kernel satisfies 
	\[P^{(q)}_k(\,\cdot\,,\,\cdot\,)\in\mathcal{C}^{\infty}(M\times M,\,L^k\otimes (T^{*0,q}M\boxtimes (T^{*0,q}M)^*)\otimes (L^k)^*) \]
	and along the diagonal we have a smooth section
	\[z \mapsto P^{(q)}_k(z,\,z)\in \mathrm{End}(T^{*0,q}_zM) \]
	called local density of states.
	
	In the standard situation, when the line bundle is positive, we consider $q=0$ and we define an embedding by setting the holomorphic map
	\[\Phi_k\,:\, M \rightarrow \mathbb{CP}^{d_k}\,,\quad  z \mapsto [S_0(z):\,\cdots \,:\,S_{d_k}(z)]\]
	where $[S_0(z):\,\cdots \,:\,S_{d_k}(z)]$ denotes the line through $(S_0(z),\cdots,\,S_{d_k}(z))$ as defined in a local holomorphic frame and where $S_0,\,\dots,\,S_{d_k}$, is an orthonormal basis of the vector space $H^0(M,L^{\otimes k})$. Since all the components transform by the same scalar under a change of frame and since
	\[P^{(0)}_k(z,\,z) = \sum_{j=0}^{d_k} |S_j(z)|^2_{h^{L^{\otimes k}}}=O(k^n)\,, \]
	it is easy to see that the sections $S_j$ do not share common zeros and the map $\Phi_k$ is well defined. Notice that for $(0,q)$-forms we have in local coordinates an expression such as \eqref{eq:form}, so we need to modify this construction as follows. Consider $L_1,\dots,L_n\in T^{1,0}M$ be an orthonormal basis of $T^{1,0}M$ with respect to $\langle\,\cdot\,|\,\cdot\,\rangle$ so that  $R^L(L_j,L_k)=\lambda_j\,\delta_{j,k}$ with $\lambda_1<0,\dots,\lambda_{n_-}<0$ and $\lambda_{n_-+1}>0,\dots,\lambda_{n}>0$. We always assume $q=n_-$ and we denote by $J_0$ the multi-index $(1,2,\dots,q)$. Denote by $e_1,\,\dots,\,e_n\in T^{*\,1,0}M$ a dual basis of $L_1,\,\dots,L_n$ and by $\overline{e}_{J_0}=\overline{e}_1\wedge\dots\wedge \overline{e}_q$, we also set
	\begin{equation} \label{eq:sec}
		S_{j,J_0}=\langle\,S_j\,|\,\ol e_{J_0}\,\rangle \in\mathcal{C}^{\infty}(M,L^{\otimes k})
	\end{equation}
	where $S_0,\,\dots,\,S_{d_k}$ is an orthonormal  basis of $H^{q}(M,L^{\otimes k})$, with $\dim_{\mathbb{C}} H^{q}(M,L^{\otimes k})=d_k+1$, in view of notation \eqref{eq:sec} define a map $\Phi_k^{(q)}: M \rightarrow\mathbb C\mathbb{P}^{d_k}$ such that
	\begin{equation} \label{eq:embmap}
		\Phi_k^{(q)}(z):= [S_{0,J_0}(z)\,:\,\cdots\,:\,S_{d_k,J_0}(z)]\in \mathbb{CP}^{d_k} \qquad \text{ on }  M\,,
	\end{equation}
	we shall show that this map is well-defined if $k$ is sufficiently large. We now state the main theorem of this paper.
	
	\begin{thm} \label{thm:maintheorem}  
		Let $(L,h^L)\To M$ be a holomorphic line bundle over a compact complex manifold $M$. Let 
		$R^L$ be the curvature of $L$ induced by $h^L$. Assume that $R^L$ is non-degenerate of constant signature $(n_-,n_+)$. With the notations used above, the map
		\[\Phi^{(n_-)}_k: M\To\mathbb C\mathbb P^{d_k}\]
		given by \eqref{eq:embmap} above is an embedding and 
		\[\left\lVert\frac{1}{k}\Phi_k^{(n_-)\,*}(\omega_{FS})-\omega  \right\rVert_{\mathcal{C}^{\infty}} \rightarrow 0 \]
		as $k$ goes to infinity in the $\mathcal{C}^{\infty}$ topology, where $\omega_{FS}$ denotes the Fubini-Study metric on $\mathbb C\mathbb P^{d_k}$ and $\omega:=\frac{i}{2\pi}R^L$. 
	\end{thm}
	
	Recall that when $L$ is positive, the convergence of $\Phi_k^{(n_-)\,*}(\omega_{FS})/k$ to $\omega$ was proved in  Tian~\cite{tian}. Notice that, the pull-back of a positive form is not necessarily positive, and in fact in our setting $\omega$ is only assumed to be a closed non-degenerate $2$-form. Furthermore, let us remark that, in general, $\Phi_k^{(n_-)}$ is not holomorphic. In fact, if $\Phi_k^{(n_-)}$ was holomorphic, then $\Phi_k^{(n_-)\,*}(\omega_{FS})$ would be a K\"ahler form for each $k$.
	
	In \cite[Chapter 5.1.4]{mm} an embedding is defined into a Grassmannian using the Bergman kernel for $E\otimes L^{\otimes k}$, where $E$ is a vector bundle of rank $e$. Here we are tensoring $L^{\otimes k}$ with $T^{*0,q}M$. %It is important to observe that in chapter $4$ of \cite{mm} for computing the asymptotic expansion of
	%the Bergman kernel it is assumed a positivity condition for the curvature $R_L$ of the line bundle $L$. Thus, \cite[Theorem 5.1.15]{mm} concerning the Grassmannian embedding is proved under the same assumption since it is based on the asymptotic expansion of the Bergman kernel.
	
	Put 
	\[\widetilde T^{1,0}M:={\rm span\,}\{\overline L_1,\ldots,\overline L_{n_-}, L_{n_-+1},\cdots,L_n\}.\]
	Then, $\widetilde T^{1,0}M$ is an almost complex structure on $M$. The second main result of this paper is the following 
	
	\begin{thm}\label{t-gue220913yyd}
		With the notations and assumptions used above, the map $\Phi^{(n_-)}_k$ is pointwise asymptotically 
		holomorphic with respect to $\widetilde T^{1,0}M$ in the following sense: 
		Let $p\in M$ and take $s$ be a local holomorphic trivializing section of $L$ 
		defined on an open set $D$ of $p$ with $\abs{s}^2_{h^{L}}=e^{-2\phi}$, 
		$\phi\in\mathcal{C}^\infty(D,\mathbb R)$, $\phi(x)=O(\abs{x-p}^2)$. Write $S_{j,J_0}=s^k\otimes\tilde S_{j,J_0}$ on $D$, $\tilde S_{j,J_0}\in\mathcal{C}^\infty(D)$, $j=0,1,\ldots,d_k$. We have 
		\begin{equation}\label{e-gue220915yyd}
			\begin{split}
				&\sum^{d_k}_{j=0}|(L_t\tilde S_{j,J_0})(p)|^2=O(k^{n}),\ \ t=1,\ldots,n_-,\\
				&\sum^{d_k}_{j=0}|(\overline L_t\tilde S_{j,J_0})(p)|^2=O(k^{n}),\ \ t=n_-+1,\ldots,n.
			\end{split}
		\end{equation}
	\end{thm} 
	
	\begin{rem}\label{r-gue220918yyd}
		With the notations and assumptions used in Theorem~\ref{t-gue220913yyd}, 
		in general, we have that (see the proof of~\cite[Theorem 4.3]{hsma})
		\[\sum^{d_k}_{j=0}|(Z\hat S_{j,J_0})(p)|^2=O(k^{n+\frac{1}{2}}),\]
		where $Z\in\mathcal{C}^\infty(M,TM\otimes\mathbb C)$. The equation \eqref{e-gue220915yyd} says that we can have better estimates if $Z=L_t$, $t\in\{1,\ldots,n_-\}$ or $Z=\overline L_t$, $t\in\{n_-+1,\ldots,n\}$. 
	\end{rem}
	
	%\begin{thm}[Theorem $5.1.15$ in \cite{mm}] If $k$ is sufficiently large, when $L$ is ample we have a well-defined map 	\begin{align*}		\Phi_k: Z \rightarrow &\mathrm{Gras}(e,\,H^0(X,\, L^k\otimes E)^*)\, \\			z \mapsto &\{s\in H^0(X,\, L^k\otimes E)\,:\, s(z)=0 \}\,,		\end{align*}		where $e$ is the rank of $E$.	\end{thm}
	%\begin{proof}		First we notice that the Bergman kernel $P_k(z, z)$ for $E \otimes L^{\otimes k}$ associated to $g_{TZ}$, $h_L$, $h_E$ when $L$ is positive have uniformly for $x\in X$ the following asymptotic expansion 		\begin{equation} \label{eq:bk}			P_k(z,z)= \mathbf{b_0}(z)\,\mathrm{Id}_E k^n + O(k^{n-1})\,,\quad \text{ with } \mathbf{b_0}(z)> 0\,.		\end{equation}		Thus, the evaluation map $H^0(X,\, L^k\otimes E)\rightarrow (L^k\otimes E)_x$ is surjective if $k$ is large enough.	\end{proof}
	
	\section{Preliminaries}
	\label{sec:pre}
	
	Let $E$ be a $\mathcal{C}^\infty$ vector bundle over $M$. The fiber of $E$ at $x\in M$ will be denoted by $E_x$.
	Let $F$ be another vector bundle over $X$. We write
	$F\boxtimes E^*$ to denote the vector bundle over $M\times M$ with fiber over $(x, y)\in M\times M$
	consisting of the linear maps from $E_y$ to $F_x$.
	
	Let $D$ be an open set of $M$. 
	The spaces of
	smooth sections of $E$ over $D$ and distribution sections of $E$ over $D$ will be denoted by $\mathcal{C}^\infty(D, E)$ and $\mathcal{D}'(D, E)$ respectively.
	Let $\mathcal{E}'(D, E)$ be the subspace of $\mathcal D'(D, E)$ whose elements have compact support in $D$. 
	Put $\mathcal{C}^\infty_c(D,E):=\mathcal{C}^\infty(D,E)\bigcap\mathcal E'(D, E)$.
	We write $\Omega^{0,q}(D,E)$ and $\Omega^{0,q}(D)$ to denote $\mathcal{C}^\infty(D,T^{*0,q}M\otimes E)$ and 
	$\mathcal{C}^\infty(D,T^{*0,q}M)$ respectively. Put $\Omega^{0,q}_c(D,E):=\Omega^{0,q}(D,E)\cap\mathcal{E}'(D,E)$, 
	$\Omega^{0,q}_c(D):=\Omega^{0,q}(D)\cap\mathcal{E}'(D)$. 
	%For $m\in\R$, we let $W^m(M,E)$ denote the Sobolev space of order $m$ of sections of $E$ and let $W_{\rm comp}^m(M, E)$ denote the Sobolev space
	%of order $m$ of sections of $E$ with compact support of $M$. Let $Y\Subset M$ be a relatively compact open set of $M$. Put
	%\[
	%W^m_{\rm loc\,}(Y, E)=\big\{u\in\mathcal D'(Y, E):\, \varphi u\in W_{\rm comp}^m(Y, E),
	%      \,\forall\varphi\in\mathcal{C}^\infty_c(Y)\big\}.\]
	Let $W_{1}, W_{2}$ be bounded open subsets of $\mathbb R^{n_{1}}$ and $\mathbb R^{n_{2}}$, respectively. Let $E$ and $F$ be complex or real vector bundles over $W_{1}$ and $W_{2}$, respectively. Let $s_{1}, s_{2}\in\mathbb R$ and $n_{0}\in\mathbb Z$. For a $k$-dependent continuous function $F_{k}: \mathcal{C}^\infty_c(W_{1}, E)\rightarrow\mathcal{D}'(W_{2}, F)$ with distribution kernel 
	$F_k(x,y)\in\mathcal{C}^\infty(W_2\times W_1, F\boxtimes E^*)$. 
	We write $F_k=O(k^{-\infty})$ on $W_2\times W_1$ or $F_k(x,y)=O(k^{-\infty})$ on $W_2\times W_1$ if for every compact set $K\Subset W_2\times W_1$, every $m\in\mathbb N\bigcup\{0\}$ and every $N>0$, there is a constant $C_{N,m,K}>0$ independent of $k$ such that
	\[\norm{F_k(x,y)}_{\mathcal{C}^m(K,F\boxtimes E^*)}\leq C_{N,m,K}k^{-N},\]
	for every $k\gg1$. Let $G_{k}: \mathcal{C}^\infty_c(W_{1}, E)\rightarrow\mathcal{D}'(W_{2}, F)$ be another $k$-dependent continuous operator with distribution kernel $G_k(x,y)\in\mathcal{C}^\infty(W_2\times W_1, F\boxtimes E^*)$. 
	We write $F_k=G_k+O(k^{-\infty})$ on $W_2\times W_1$ or $F_k(x,y)=G_k(x,y)+O(k^{-\infty})$ on $W_2\times W_1$ if $F_k-G_k=O(k^{-\infty})$ on $W_2\times W_1$. 
	
	We recall the definition of the semi-classical symbol spaces 
	
	\begin{defn} \label{d-gue140826}
		Let $W$ be an open set in $\mathbb R^N$. 
		We let $S(1;W)$ be the set of $a\in\mathcal{C}^\infty(W)$ such that for every $\alpha\in(\mathbb N\bigcup\{0\})^{N}$, there exists $C_\alpha>0$ such that 
		\[\abs{\pr^\alpha a(x)}\leq C_\alpha,\] for every $x\in W$.
		%Let
		%$S(1;W)=S(1)$ be the set of
		%$a\in \cC^\infty(W)$ such that for every $\alpha\in\mathbb N^N_0$, there
		%exists $C_\alpha>0$, such that $\abs{\pr^\alpha_xa(x)}\leq
		%C_\alpha$ on $W$.
		%\begin{gather*}
		%S(m;W):=\Big\{a\in\cC^\infty(W)\,|\, \text {\rm for all}\alpha\in\mathbb N^N_0:
		%\sup_{x\in W}\abs{\pr^\alpha a(x)}<\infty\Big\},\\
		%S^0_{{\rm loc\,}}(f;W):=\Big\{(a(\cdot,k))_{k\in\R}\,|\, \text {\rm for all}~\alpha\in\mathbb N^N_0,
		%\text {\rm for all}~ \chi\in\cC^\infty_0(W)\,:\:\chi a\in S(f;W)\Big\}\,.
		%\end{gather*}
		%The space $S_{{\rm loc\,}}(1;W)$ is the set of sequences $a=(a(\cdot,k))$ in $\cC^\infty(W)$
		%with the property that for any $\chi\in\cC^\infty_0(W)$ we have
		%%\[
		%$\sup_{k\in\N}\sup_{x\in W}\abs{\pr^\alpha a(x,k)}<\infty$\,.
		%%\]
		%If $a=a(x,k)$ depends on $k\in]1,\infty[$, we say that
		%$a(x,k)\in S_{{\rm loc\,}}(1;W)=S_{{\rm loc\,}}(1)$ if $\chi(x)a(x,k)$ is uniformly bounded
		%in $S(1)$ when $k$ varies in $]1,\infty[$, for any $\chi\in\cC^\infty_0(W)$.
		For $m\in\mathbb R$, let
		\[S^m(1)=S^m(1;W):=\Big\{(a(\cdot,k))_{k\in\mathbb R}\,|\,(k^{-m}a(\cdot,k))\in S^0(1;W)\Big\}.\]
		%&S^m_{{\rm loc}}(f):=S^m_{{\rm loc}}(f;W)=\Big\{(a(\cdot,k))_{k\in\R}\,|\,(k^{-m}a(\cdot,k))\in S^0_{{\rm loc\,}}(f;W)\Big\}\,.
		
		%For $m\in\R$, we put $S^m_{{\rm loc}}(1;W)=\{(a(\cdot,k)):(k^{-m}a(\cdot,k))\in S_{{\rm loc\,}}(1;W)\}$.
		%For $m\in\R$, we put $S^m_{{\rm loc}}(1;W)=S^m_{{\rm loc}}(1)=k^mS_{{\rm loc\,}}(1)$.

		Consider a sequence $a_j\in S^{m_j}(1)$, $j\in\mathbb N\cup\{0\}$, where $m_j\searrow-\infty$,
		and let $a\in S^{m_0}(1)$. We say that
		\[
		a(\cdot,k)\sim
		\sum\limits^\infty_{j=0}a_j(\cdot,k)\:\:\text{in $S^{m_0}(1)$},
		\]
		if for every
		$\ell\in\mathbb N\cup\{0\}$ we have $a-\sum^{\ell}_{j=0}a_j\in S^{m_{\ell+1}}(1)$. 
		For a given sequence $a_j$ as above, we can always find such an asymptotic sum
		$a$, which is unique up to an element in $S^{-\infty}(1)=S^{-\infty}(1;W):=\cap _mS^m(1)$.
		
		We say that $a(\cdot,k)\in S^{m}(1)$ is a classical symbol on $W$ of order $m$ if
		\begin{equation} \label{e-gue13628I}
			a(\cdot,k)\sim\sum\limits^\infty_{j=0}k^{m-j}a_j\: \text{in $S^{m}(1)$},\ \ a_j(x)\in
			S^0(1),\ j=0,1\ldots.
		\end{equation}
		The set of all classical symbols on $W$ of order $m$ is denoted by
		$S^{m}_{{\rm cl\,}}(1)=S^{m}_{{\rm cl\,}}(1;W)$. 
		
		Similarly, we define $S^m(1;Y,E)$, $S^{m}_{{\rm cl\,}}(1;Y,E)$ in the standard way, where $Y$ is a smooth manifold and $E$ is a vector bundle over $Y$. 
	\end{defn}
	
	We come back to our situation. We will use the same notations as Section~\ref{s-gue220919yyd}. 
	We will identify the curvature form $R^L$ with the Hermitian matrix 
	\[\dot{R}^L\in\mathcal{C}^{\infty}(M,\mathrm{End}(T^{1,0}M))\,, \quad \langle\,R^L(z)\,,\,U\wedge\overline{V} \,\rangle = \langle\,\dot{R}^L(z)U\,|\,V\,\rangle,\]
	for every $U$ and $V$ in $T^{1,0}_zM$, $z\in M$. We denote by $W$ the subbundle of rank $q=n_-$ of $T^{1,0}M$ generated by the eigenvectors corresponding to negative eigenvalues of $\dot{R}^L$. Then, ${\rm det\,}\ol{W}^*:=\Lambda^q\ol{W}^*$ is a rank one subbundle, where $\ol{W}^*$ is the dual bundle of the complex conjugate bundle of $W$ and $\Lambda^q\ol{W}^*$ is the vector space of all finite sum of $v_1\wedge\cdots\wedge v_q$, $v_1,\ldots,v_q\in\ol{W}^*$. We denote by 
	$I_{{\rm det\,}\ol{W}^*}$ the orthogonal projection from $T^{*0,q}M$ onto ${\rm det\,}\ol{W}^*$. For $x\in M$, let $\det\dot{R}^L(x):=\mu_1(x)\cdots\mu_n(x)$, where $\mu_j(x)$, $j=1,\ldots,n$, are eigenvalues of $\dot{R}^L(x)$. 
	
	Let $s$ and $s_1$ be local holomorphic trivializing sections of $L$ defined on open sets $D\subset M$ and $D_1\subset M$ respectively, $\abs{s}^2_{h^L}=e^{-2\phi}$, $\abs{s_1}^2_{h^L}=e^{-2\phi_1}$, $\phi\in\mathcal{C}^\infty(D,\mathbb R)$, $\phi\in\mathcal{C}^\infty(D_1,\mathbb R)$. The localization of $P^{(q)}_k$ with respect to $s$, $s_1$ are given by 
	\[\begin{split}
		P^{(q)}_{k,s,s_1}: \Omega^{0,q}_c(D)&\To\Omega^{0,q}(D_1),\\
		u&\To s^{-k}_1e^{-k\phi_1}(P^{(q)}_k(s^ke^{k\phi}u)).
	\end{split}\]
	Let $P^{(q)}_{k,s,s_1}(x,y)\in\mathcal{C}^\infty(D_1\times D,T^{*0,q}M\boxtimes(T^{*0,q}M)^*)$ be the distribution kernel of $P^{(q)}_{k,s,s_1}$. When $D=D_1$, $s=s_1$, we write $P^{(q)}_{k,s}:=P^{(q)}_{k,s,s}$, 
	$P^{(q)}_{k,s}(x,y):=P^{(q)}_{k,s,s}(x,y)$. We recall the following result~\cite[Theorem 4.11]{hsma} 
	
	\begin{thm}\label{t-gue220914yyd} 
		With the notations and assumptions used above and recall that we let $q=n_-$. Let $s$ be a local holomorphic trivializing section of $L$ defined on an open set $D\subset M$, $\abs{s}^2_{h^L}=e^{-2\phi}$. We have 
		\begin{equation}\label{e-gue220914yyd}
			P^{(q)}_{k,s}(x,y)=e^{ik\Psi(x,y)}b(x,y,k)+O(k^{-\infty})\ \ \mbox{on $D\times D$},
		\end{equation}
		where $b(x,y,k)\in S^n_{{\rm cl\,}}(1;D\times D,T^{*0,q}M\boxtimes(T^{*0,q}M)^*)$, 
		$b(x,y,k)\sim\sum_{j=0}b_j(x,y)k^{n-j}$ in $S^n_{{\rm cl\,}}(1;D\times D,T^{*0,q}M\boxtimes(T^{*0,q}M)^*)$, $b_j(x,y)\in\mathcal{C}^\infty(D\times D,T^{*0,q}M\boxtimes(T^{*0,q}M)^*)$, $j=0,1,\ldots$, 
		\begin{equation}\label{e-gue220914yydI}
			b_0(x,x)=(2\pi)^{-n}I_{{\rm det\,}\ol{W}^*}\abs{\det\dot{R}^L(x)},\ \ \mbox{for every $x\in M$}, 
		\end{equation}
		and $\Psi(x,y)\in\mathcal{C}^\infty(D\times D)$, 
		\begin{equation}\label{e-gue220914yydII}
			\begin{split}
				&\Psi(x,y)=-\ol\Psi(y,x),\\
				&\mbox{there is a $c>0$ such that ${\rm Im\,}\Psi(x,y)\geq c\abs{x-y}^2$, for all $x,y\in D$},\\
				&\mbox{$\Psi(x,y)=0$ if and only of $x=y$}. 
			\end{split}
		\end{equation}
		Fix $p\in D$, take $z=(z_1,\ldots,z_n)$ be local coordinates of $M$ defined on $D$ and $s$ so that 
		\begin{equation}\label{e-gue220915yydI}
			\begin{split}
				&z(p)=0,\\
				&\langle\,\frac{\pr}{\pr z_j}\,|\,\frac{\pr}{\pr z_t}\,\rangle=\delta_{j,t}+O(\abs{z}),\ \ j,t=1,\ldots,n,\\
				&\phi(z)=\sum^n_{j=1}\lambda_j\abs{z_j}^2+O(\abs{z}^3),
			\end{split}
		\end{equation}
		then 
		\begin{equation}\label{e-gue220914yydIII}
			\Psi(z,w)=i\sum^n_{j=1}\abs{\lambda_j}\abs{z_j-w_j}^2+i\sum^n_{j=1}\lambda_j(\ol z_jw_j-z_j\ol w_j)+O(\abs{(z,w)}^3). 
		\end{equation}
	\end{thm}

	\section{Proof of Theorem \ref{thm:maintheorem}}
	
	We always assume $q=n_-$. In view of Theorem~\ref{t-gue220914yyd}, we have 
	%By Theorem $1.1$ in \cite{hsma}, the Bergman kernel for $\wedge^qT^{*(0,1)}Z\otimes L^{\otimes k}$, where $q$ is the number of negative eigenvalues of $\dot{R}_L$, is
	\begin{equation} \label{eq:expan}
		P_k^{(q)}(z,z)= \mathbf{b_0}(z)\,\mathrm{I}_{\det \overline{W}^*} k^n + O(k^{n-1})\,,
	\end{equation}
	with $\mathbf{b_0}(z)> 0$. By \eqref{eq:expan} follows that the map $\Phi^{(q)}_k$ is well-defined.
	
	We shall now prove that $\Phi_k^{(q)}$ is injective. By contrary, up to consider a subsequence, suppose that we can find $x_k,\,y_k$ in $M$, such that $x_k\neq y_k$ and $\Phi_k^{(q)}(x_k)=\Phi_k^{(q)}(y_k)$, for each $k$. Thus, by noticing
	\begin{equation}\label{e-gue220920yyd}
		\begin{split}
			&\forall k\,:\, \Phi_k^{(q)}(x_k)=\Phi_k^{(q)}(y_k) \iff \\
			&\forall k,\,\,\exists \lambda_k\in \mathbb{C}^*\,:\, (S_{0,J_0}(x_k),\,\dots,\,S_{d_k,J_0}(x_k))= \lambda_k\cdot (S_{0,J_0}(y_k),\,\dots,\,S_{d_k,J_0}(y_k)) \end{split}\end{equation}
	we can also assume that for each harmonic $(0,q)$-forms $g_k\in H^q(M,L^{\otimes k})$, we have
	\begin{equation}\label{eq:cont}
		\abs{g_{k,J_0}(x_k)}^2_{h^{L^{\otimes k}}} \geq\abs{g_{k,J_0}(y_k)}^2_{h^{L^{\otimes k}}},\ \ \mbox{for all $k\gg1$}.
	\end{equation} 
	Assume that $\lim_{k\To+\infty}x_k=p$, $\lim_{k\To+\infty}y_k=q$, $p\neq q$. 
	Let $s$ and $s_1$ be local holomorphic trivializing sections of $L$ defined on open sets $D\subset M$ and $D_1\subset M$ respectively, $p\in D$, $q\in D_1$, $\ol D\cap\ol D_1=\emptyset$. 
	From~\cite[Lemma 3.7]{hsma1}, we see that 
	\begin{equation}\label{e-gue220914ycd}
		P^{(q)}_{k,s,s_1}(x,y)=O(k^{-\infty})\ \ \mbox{on $D\times D_1$}.
	\end{equation}
	From Theorem~\ref{t-gue220914yyd} and \eqref{e-gue220914ycd}, we can find $g_k\in H^q(M,L^{\otimes k})$ such that 
	\begin{equation}\label{e-gue220919yyd}
		\abs{g_{k,J_0}(x_k)}^2_{h^{L^{\otimes k}}}<\abs{g_{k,J_0}(y_k)}^2_{h^{L^{\otimes k}}},\ \ \mbox{for all $k\gg1$}.
	\end{equation}
	From \eqref{eq:cont} and \eqref{e-gue220919yyd}, we get a contradiction. Therefore, $x_k$ and $y_k$ must converge to the same point $p\in M$. We can distinguish two cases: either $\limsup_k \sqrt{k}\lvert x_k -y_k \rvert=C>0$, $C\in ]0,\infty ]$ or $\limsup_k \sqrt{k}\lvert x_k -y_k \rvert=0$. Here the distance $\lvert x_k-y_k\rvert$, $x_k, y_k\in M$, is meant with respect to the Riemannian structure  ${g}=(\langle\cdot,\cdot\rangle+\overline{\langle \cdot,\cdot\rangle})/2$ on $TM$.
	
	Let $s$ be local holomorphic trivializing section of $L$ defined on an open set $D\subset M$, $p\in D$, $\abs{s}^2_{h^L}=e^{-2\phi}$. 
	Let us study the first case. For every $y_k\in D$, let us set \[g_k(\cdot):=P_{k}^{(q)}(\cdot,\,y_k)\ol e_{J_0}(y_k)\in H^q(M,L^{\otimes k}).\] 
	On $D$, write $g_k=s^k\otimes\tilde g_k$, $\tilde g_k\in\Omega^{0,q}(D)$. Put $\hat g_k:=e^{-k\phi}\tilde g_k$. 
	In view of Theorem~\ref{t-gue220914yyd}, we have
	\[\hat g_k(x)=e^{ik\, \Psi(x, y_k)}b(x, y_k,k)\ol e_{J_0}(y_k)+O(k^{-\infty})\ \ \mbox{on $D$}.\]
	Let $\hat g_{k,J_0}(x):=\langle\,\hat g_k(x)\,|\,\ol e_{J_0}(x)\,\rangle\in\mathcal{C}^\infty(D)$. 
	Notice that
	\[\begin{split}
		&\abs{\hat g_{k,J_0}(x_k)}^2=\abs{e^{ik\, \Psi(x_k,y_k)}\langle\,b(x_k, y_k,k) \ol e_{J_0}(y_k)\,|\,\ol e_{J_0}(x_k)\,\rangle}^2+O(k^{-\infty})\\ &\leq  e^{-2k\,\mathrm{Im}\Psi({x}_k, {y}_k)}\abs{\langle\,b({x}_k,{y}_k,k)\ol e_{J_0}(y_k)\,|\,e_{J_0}(x_k)\,\rangle}^2+O(k^{-\infty})\end{split}\]
	and
	\[\abs{\hat g_{k,J_0}(y_k)}^2=\abs{\langle\,  b({y}_k,{y}_k,k)\overline{e}_{J_0}(y_k)\,|\,e_{J_0}(x_k)\,\rangle}^2+O(k^{-\infty}).\]
	Thus, by recalling that $x_k,\,y_k \rightarrow p$, we get
	\[\lim_{k\To+\infty}  \frac{\abs{\hat g_{k,J_0}(x_k)}^2}{k^{2n}} \leq  e^{-2c\,C^2} \abs{b_0(p,p)}^2  <  \abs{b_0(p,p)}^2 = \lim_{k\To+\infty}  \frac{\abs{\hat g_{k,J_0}(y_k)}^2}{k^{2n}} \]
	which contradicts \eqref{eq:cont}, where $c>0$ is a constant independent of $k$. 
	
	Now, let us focus on the second case. 
	Put $P^{(q)}_{k,s,J_0,J_0}(x,y):=\langle\,P^{(q)}_{k,s}(x,y)e_{J_0}(y)\,|\,e_{J_0}(x)\,\rangle$ 
	and let $P^{(q)}_{k,s,J_0,J_0}(x):=P^{(q)}_{k,s,J_0,J_0}(x,x)$.
	Let us set
	\[f_k(t):=\frac{\abs{P_{k,s,J_0,J_0}^{(q)}(tx_k+(1-t)y_k,y_k)}^2} {P^{(q)}_{k,s,J_0,J_0}(tx_k+(1-t)y_k)P^{(q)}_{k,s,J_0,J_0}(y_k)}.   \]
	From Theorem~\ref{t-gue220914yyd}, it is then easy to prove that
	\begin{equation}f_k(t)\sim e^{-2k\,\mathrm{Im}\Psi(t{x}_k+(1-t){y}_k, {y}_k)}\cdot\left[1+\frac{1}{\sqrt{k}}R_k(t)\right] \label{eq:aa}
	\end{equation}
	where $R_k=O(1)$ and the symbol $\sim$ means ``has the same asymptotic as''. For ease of notation, let us set $F_k(t):=-2k\,\mathrm{Im}\Psi(t{x}_k+(1-t){y}_k, {y}_k)$ for given $x_k$, $y_k$. Now, by applying the following formulas
	\[\lvert F'_k(t) \rvert=\lvert\langle -2k\,\mathrm{Im}\Psi'_x(tx_k+(1-t)y_k,y_k), x_k-y_k \rangle\rvert \leq \frac{1}{c_0}k \lvert x_k-y_k\rvert^2  \]
	and
	\[F''_k(t)=\langle -2k\,\mathrm{Im}\Psi''_x(tx_k+(1-t)y_k, y_k), x_k-y_k \rangle < -c_0k\,\lvert x_k-y_k\rvert^2\]
	for the computation of the second derivative of $f_k(t)$ in equation \eqref{eq:aa}, here $c_0$ is a positive constant, and by using the assumption $\limsup_k \sqrt{k}\lvert x_k -y_k \rvert=0$, after some computations, we obtain that
	\begin{equation}\limsup_k\frac{f''(t)}{k\,\lvert x_k-y_k\rvert^2}< 0  \label{eq:lim}
	\end{equation}
	for each $t$.
	
	We now show that if we apply the assumption \eqref{e-gue220920yyd} for computing the $\limsup$ appearing in \eqref{eq:lim} we get a different result. By Cauchy-Schwartz inequality we have $0\leq f_k(t)\leq 1$, for any $t\in [0,1]$ and $f_k(0)=f_k(1)=1$. Thus, for each $k$, there is a $t_k\in [0, 1]$ such that $f''(t_k)=0$. Hence,
	\[\limsup_k\frac{f''(t)}{k\,\lvert x_k-y_k\rvert^2}\geq 0 \]
	which contradicts \eqref{eq:lim}. 
	
	Eventually, we shall prove the last part of the theorem concerning pull-back of the Fubini-Study form, as a corollary we get that $\mathrm{d}\Phi^{(q)}_k$ is injective and thus $\Phi^{(q)}_k$ is an embedding (see~\cite[Lemma 5.16]{mm}). Recall that the Fubini Study K\"ahler form on $\mathbb{CP}^{m}$ is given in homogeneous coordinates $[w_0:\cdots:w_{m}]$ by
	\begin{equation}\omega_{\mathrm{FS}}:=\frac{i}{2\pi}\partial\overline{\partial}\log\left(\sum_{j=0}^m\lvert w_i \rvert^2\right)\,. \label{eq:fs} 
	\end{equation}
	Let us apply \eqref{eq:fs} to our setting; local holomorphic section $s$ of $L$ over an open set $D$ in $M$. It induces a section $s^k$ of the $k$-th tensor power of the line bundle $L^{\otimes k}$ on $D$. By recalling \eqref{eq:sec}, we have $S_{j,J_0}=f_{j,J_0}(z)\otimes s^k$ for a smooth function $f_{j,J_0}$ on $D$, $j=0,1,\ldots,d_k$. By definition \eqref{eq:embmap}, we get
	\begin{align*}
		\frac{1}{k}\Phi^{(q)\,*}_k(\omega_{\mathrm{FS}})_z&=\frac{i}{2\pi k}\partial\overline{\partial}\log\left(\sum_{j=0}^{d_k}\lvert f_{j,J_0}(z) \rvert^2\right)=\frac{i}{2\pi k}\partial\overline{\partial}\log\left(\sum_{j=0}^{d_k}\frac{\abs{S_{j,J_0}(z)}^2_{h^{L^{\otimes k}}}}{\abs{s^k(z)}^2_{h^{L^{\otimes k}}}}\right) \\
		&=\frac{i}{\pi}\pr\ddbar\phi+\frac{i}{2\pi k}\partial\overline{\partial}\log\left(\sum_{j=0}^{d_k}{\abs{S_{j,J_0}(z)}^2_{h^{L^{\otimes k}}}}\right)\,. 
	\end{align*}
	Thus, the statement of the theorem follows from
	\[\norm{\pr\ddbar\log P^{(q)}_{k,s,J_0,J_0}(x,x)}_{\mathcal{C}^\infty}=O(1)\]
	in view of the asymptotic expansion of $P^{(q)}_{k,s}(x,x)$. 
	
	\section{Proof of Theorem~\ref{t-gue220913yyd}}\label{s-gue220915yyd}
	
	Fix $p\in M$ and let $s$ be a local trivializing section of $L$ defined on an open set 
	$D$ of $p$, $\abs{s}^2_{h^L}=e^{-2\phi}$. Let $z=(z_1,\ldots,z_n)=(x_1,\ldots,x_{2n})$, $z_j=x_{2j-1}+ix_{2j}$, $j=1,\ldots,n$, be local coordinates of $M$ defined on $D$. We take $z$ and $s$ so that \eqref{e-gue220915yydI} hold. Assume that $\lambda_1<0,\ldots,\lambda_{n_-}<0$, $\lambda_{n_-+1}>0,\ldots,\lambda_n>0$. We have $L_j=\frac{\pr}{\pr z_j}$ at $p$, $j=1,\ldots,n$. 
	We need 
	
	\begin{lem}\label{l-gue220915yyd}
		Let $u\in H^q(M,L^{\otimes k})$ with $\norm{u}_k=1$. On $D$, write $u=s^k\otimes\tilde u$, $\tilde u\in\Omega^{0,q}(D)$. Put $\tilde u_{J_0}(x):=\langle\,\tilde u(x)\,|\,\ol e_{J_0}(x)\,\rangle\in\mathcal{C}^\infty(D)$. 
		There is a constant $C>0$ independent of $k$ and $u$ such that 
		\begin{equation}\label{e-gue220915ycda}
			\begin{split}
				&\abs{(L_t\tilde u_{J_0})(p)}^2\leq Ck^{n},\ \ t=1,\ldots,n_-,\\
				&\abs{(\ol L_t\tilde u_{J_0})(p)}^2\leq Ck^{n},\ \ t=n_-+1,\ldots,n.
			\end{split}
		\end{equation}
	\end{lem}
	
	\begin{proof}
		Put $Z_t=L_t$, $t=1,\ldots,n_-$, $Z_t=\ol L_t$, $t=n_-+1,\ldots,n$. Fix $t\in\{1,\ldots,n\}$. 
		From Theorem~\ref{t-gue220914yyd} and off-diagonal expansion of $P^{(q)}_k$ (see \eqref{e-gue220914ycd}), we have 
		\begin{equation}\label{e-gue220915ycd}
			\abs{(Z_t\tilde u_{J_0})(p)-\int e^{ik\Psi(0,y)-k\phi(y)}\Bigr(r_0(y)k^{n+1}+r_1(y,k)\Bigr)\chi(y)\tilde u(y)\mathrm{dV}_M(y)}\leq C_Nk^{-N},
		\end{equation}
		for every $N>0$, where $C_N>0$ is a constant independent of $k$, $u$ and $r_0(y)=O(\abs{y}^2)$, 
		$r_1(y,k)\in S^n_{{\rm cl\,}}(1;D,\mathbb C\boxtimes(T^{*0,q}M)^*)$, $\chi\in\mathcal{C}^\infty_c(D)$, $\chi=1$ near $0$. From ${\rm Im\,}\Psi(0,y)\sim\abs{y}^2$ and some straightforward computations, we can check that 
		\begin{equation}\label{e-gue220915ycdI}
			\begin{split}
				&\abs{\int e^{ik\Psi(0,y)-k\phi(y)}\Bigr(r_0(y)k^{n+1}+r_1(y,k)\Bigr)\chi(y)\tilde u(y)\mathrm{dV}_M(y)}\\
				&\leq C\abs{\int e^{-ck\abs{y}^2}(k^{2n+2}\abs{y}^4+k^{2n})dy}^{\frac{1}{2}}\leq\hat Ck^n,
			\end{split}
		\end{equation}
		where $C>0$, $c>0$, $\hat C>0$ are constants independent of $k$ and $u$. 
		From \eqref{e-gue220915ycd} and \eqref{e-gue220915ycdI}, we get \eqref{e-gue220915ycda}.
	\end{proof}
	
	We can now prove Theorem~\ref{t-gue220913yyd}. 
	We may assume that $\sum^{d_k}_{j=0}\abs{(Z_t\tilde S_{j,J_0})(p)}^2\neq0$. 
	Let 
	\[u:=\frac{\sum^{d_k}_{j=0}S_j\ol{(Z_t\tilde S_{j,J_0})(p)}}{\sqrt{\sum^{d_k}_{j=0}\abs{(Z_t\tilde S_{j,J_0})(p)}^2}}.\]
	Then, $u\in H^q(M,L^{\otimes k})$ with $\norm{u}_k=1$. From Lemma~\ref{l-gue220915yyd},
	we get 
	\[\abs{(Z_t\tilde u_{J_0})(p)}^2=\sum^{d_k}_{j=0}\abs{(Z_t\tilde S_{j,J_0})(p)}^2\leq Ck^{n}.\]
	The theorem follows. 
	
	\bigskip
	
	\textbf{Acknowledgements:} This project was started during the first author’s postdoctoral fellowship at the National Center for Theoretical Sciences in Taiwan; we thank the Center for the support. Chin-Yu Hsiao was partially supported by Taiwan Ministry of Science and Technology projects  108-2115-M-001-012-MY5, 109-2923-M-001-010-MY4.

\end{document}